\documentclass[11pt]{amsart}

\usepackage{amsmath}
\usepackage{amsfonts}
\usepackage{amssymb}
\usepackage{amsthm}

\usepackage{mathrsfs}
\usepackage[colorlinks=true,urlcolor=blue,linkcolor=blue,citecolor=blue,pagebackref,breaklinks]{hyperref}
\usepackage{cite}

\usepackage{bm}
\usepackage{dsfont}
\usepackage{upgreek}

\usepackage{fullpage}

\usepackage{graphicx}
\usepackage[all,cmtip]{xy}
\usepackage{enumerate}
\usepackage{blkarray}
\usepackage{array}
\usepackage{leftidx}
\usepackage{mathdots}
\usepackage{arydshln}
\usepackage{stmaryrd}

\newenvironment{psm}
{\left(\begin{smallmatrix}}
{\end{smallmatrix}\right)}

\numberwithin{equation}{subsection}

\newtheorem{thm}{Theorem}[subsection]
\newtheorem{prop}[thm]{Proposition}

\newtheorem{coro}[thm]{Corollary}
\newtheorem{defn}[thm]{Definition}

\theoremstyle{remark}
\newtheorem{rmk}[thm]{Remark}

\newtheorem*{Ac}{Acknowledgements}

\newcommand{\mr}[1]{\mathrm{#1}}

\newcommand{\bA}{\mathbb{A}}

\newcommand{\bC}{\mathbb{C}}

\newcommand{\bH}{\mathbb{H}}

\newcommand{\bQ}{\mathbb{Q}}
\newcommand{\bR}{\mathbb{R}}

\newcommand{\bZ}{\mathbb{Z}}

\newcommand{\cD}{\mathcal{D}}

\newcommand{\cI}{\mathcal{I}}

\newcommand{\cS}{\mathcal{S}}

\newcommand{\fc}{\mathfrak{c}}

\newcommand{\fq}{\mathfrak{q}}

\newcommand{\fs}{\mathfrak{s}}

\newcommand{\fH}{\mathfrak{H}}

\newcommand{\fQ}{\mathfrak{Q}}

\newcommand{\sC}{\mathscr{C}}

\DeclareMathAlphabet{\mathpzc}{OT1}{pzc}{m}{it}

\newcommand{\pW}{\mathpzc{W}}

\newcommand{\ba}{\bm{a}}

\newcommand{\bz}{\bm{z}}

\newcommand{\ut}{\protect\underline{t}}

\newcommand{\utau}{\protect\underline{\uptau}}

\DeclareMathOperator{\GL}{GL}

\DeclareMathOperator{\U}{U}
\DeclareMathOperator{\OO}{O}

\DeclareMathOperator{\Sp}{Sp}
\DeclareMathOperator{\SO}{SO}

\newcommand{\cont}{\mathrm{cont}}

\newcommand{\ddet}{\mathrm{det}}

\newcommand{\sgn}{\mathrm{sgn}}

\newcommand{\Tr}{\mathrm{Tr}}

\newcommand{\vol}{\mathrm{vol}}

\DeclareMathOperator{\Hom}{Hom}
\DeclareMathOperator{\Ind}{Ind}

\DeclareMathOperator{\Lie}{Lie}

\DeclareMathOperator{\Sym}{Sym}

\newcommand{\lhra}{\ensuremath{\lhook\joinrel\longrightarrow}}
\newcommand{\lra}{\longrightarrow}
\newcommand{\ra}{\rightarrow}

\newcommand{\ol}{\overline}
\newcommand{\wh}{\widehat}
\newcommand{\wt}{\widetilde}

\newcommand{\bid}{\mathbf{1}}
\newcommand{\mvw}{\vartheta}

\newcommand{\MC}{\mathrm{MC}}
\newcommand{\hol}{\mathrm{hol}}
\newcommand{\inv}{\mathrm{inv}}

\makeatletter
\newcommand{\leqnomode}{\tagsleft@true\let\veqno\@@leqno}
\newcommand{\reqnomode}{\tagsleft@false\let\veqno\@@eqno}
\makeatother

\makeatletter
\def\l@section{\@tocline{1}{0pt}{1pc}{}{}}
\def\l@subsection{\@tocline{2}{0pt}{1pc}{4.6em}{}}
\def\l@subsubsection{\@tocline{3}{0pt}{1pc}{7.6em}{}}
\renewcommand{\tocsection}[3]{%
  \indentlabel{\@ifnotempty{#2}{\makebox[2.3em][l]{%
    \ignorespaces#1 #2.\hfill}}}#3}
\renewcommand{\tocsubsection}[3]{%
  \indentlabel{\@ifnotempty{#2}{\hspace*{2.3em}\makebox[2.3em][l]{%
    \ignorespaces#1 #2.\hfill}}}#3}
\renewcommand{\tocsubsubsection}[3]{%
  \indentlabel{\@ifnotempty{#2}{\hspace*{4.6em}\makebox[3em][l]{%
    \ignorespaces#1 #2.\hfill}}}#3}
\makeatother

\title{The doubling archimedean zeta integrals for $p$-adic interpolation}
\author{Zheng Liu}
\begin{document}

\begin{abstract}
We compute the archimedean doubling zeta integrals which appear in the interpolation formulas for the $p$-adic $L$-functions of Siegel modular forms constructed in \cite{SLF}.
\end{abstract}

\maketitle
\tableofcontents 
\numberwithin{equation}{subsection}

The doubling method, discovered by Garrett \cite{GaPull}, B{\"o}cherer \cite{BoDouble} and Piatetski-Shapiro--Rallis \cite{PSR},  provides an integral representation of automorphic $L$-functions for classical groups in terms of Siegel Eisenstein series. It has been vastly applied to study the properties of automorphic $L$-functions, inculding defining local $L$-factors \cite{LapidRallis,Yam}, showing the meromorphic continuation of the global $L$-functions and locating the possible poles \cite{PSR,KRPoles}, proving the algebraicity of the critical $L$-values normalized by a Petersson inner product period when the representation is algebraic \cite{Sh00,Ha81,HaSim}, and constructing the $p$-adic interpolation of those normalized critical values \cite{BS,SLF,EisWan,EHLS}.

For many applications, a key technical ingredient is about proving the desired properties of the doubling zeta integrals 
\begin{equation}\label{eq:zgenreal}
   Z_v\left(f_v(s),v_1,v_2\right)=\int_{G(F_v)}f_v(s)\left(\iota^\diamond(g,1)\right)\left<\pi_v(g)v_1,v_2\right>\,dg,
\end{equation}
at an archimedean place $v\mid \infty$, where $\left<\pi_v(g)v_1,v_2\right>$ is a matrix coefficient of the local representation at $v$ of the cuspidal automorphic representation on a classical group $G$ we are interested in, $\iota^\diamond$ is the doubling embedding of $G\times G$ into a group that doubles the size of $G$, and $f_v(s)$ is a section inside the degenerate principal series on the double sized group. In \cite{KRPoles}, it has been shown that for all $s\in\bC$ there always exists an $f_v(s)$ such that \eqref{eq:zgenreal} is nonzero; thus the possible poles of the $L$-functions are determined by those of the Siegel Eisenstein series. For arithimetic applications, the $f_v(s)$ used in \cite{KRPoles} is not good enough. One further requires that with the choice of the archimedean sections, the Siegel Eisenstein series are algebraic at the points corresponding to the critical points of the $L$-function. To this end, a rich theory of algebraic differential operators  has been developed \cite{ShDiff,ShLie,HaBun,BoDiff,Ib} when $G$ is symplectic or unitary with a Shimura variety. The argument in \cite{HaBun} shows that the algebraic differential operators, which give rise to sections with potentially nonvanishing archimedean zeta integrals, constructed through different approaches agree up to scalar. In \cite{HaSim,SLF}, the nonvanishing of the archimedean zeta integrals for those sections has been verified for the unitary and symplectic cases. When the holomorphic discrete series is of scalar weight, the holomorphic differential operators have been written down explicitly and the corresponding archimedean zeta integrals have been computed in \cite{Sh00,BS}. For vector weights, a special case is considered in \cite{GaAzint} for the unitary group $\U(a,b)$ with the requirement on the weight $\tau_1\boxtimes\tau_2$ that one of $\tau_1$ and $\tau_2$ is scalar. 

In this paper, we compute the archimedean zeta integrals for general vector weights at $s=s_0$ in the case when $G=\Sp(2n)_{/\bQ}$, $s_0$ corresponds to a critical point and the section $f_\infty(s)$ is chosen as in the construction of the $p$-adic $L$-functions in \cite{SLF}. Let us briefly recall the setting and notation {\it loc. cit}. Let $H=\Sp(4n)_{/\bQ}$ (the doubling group on which the Siegel Eisenstein series live) and $Q_H\subset H$ be the standard Siegel parabolic subgroup. Given an $(n+1)$-tuple of integers
\begin{align*}
   (\ut,k)&=(t_1,\dots,t_n,k), &t_1\geq\dots t_n\geq k\geq n+1,
\end{align*}
and supposing that the archimedean component of the irreducible cuspidal automorphic representation of $G(\bA)$ is isomorphic to $\cD_{\ut}$, the holomorphic discrete series of weight $\ut$, then in order to study the critical value at $s_0=k-n$ or $s_0=n+1-k$ of the standard $L$-function, we picked an archimedean section (denoted as $f_{\kappa,\utau,\infty}$ in \cite{SLF})
$$
   f_{k,\ut}\in\Ind^{H(\bR)}_{Q_H(\bR)}\sgn^k|\cdot|^{k-\frac{2n+1}{2}}
$$
(see \S\ref{sec:fkt} for the precise definition). Let $v_{\ut}\in\cD_{\ut}$ be the highest weight vector inside the lowest $K_G$-type and $v^\vee_{\ut}$ be its dual vector in the contragredient representation.

\begin{thm}[Theorem~\ref{thm:Main}]
\begin{align*}
   Z_\infty\left(f_{k,\ut},v^\vee_{\ut},v_{\ut}\right)=\frac{i^{-\sum_{j=1}^n t_j+nk}2^{-2\sum_{j=1}^nt_j-nk+2n^2+2n}\pi^{-\sum_{j=1}^nt_j+nk+\frac{3n^2+n}{2}}}{\Gamma_{2n}(k)\,\dim\left(\GL(n),\ut\right)}\,\prod_{j=1}^n\Gamma(t_j-j+k-n).
\end{align*}
\end{thm}

Plugging this result into the interpolation formulas of our previously constructed $p$-adic $L$-functions, we verify that the archimedean factors in the interpolation formulas agree with the conjecture of Coates--Perrin-Riou \cite{CoPerrin,Coates}. 

Unfolding the definitions in \cite[\S5]{Coates} for our case, the modified archimedean Euler factor for $p$-adic interpolation of the critical values to the right (resp. left) of the center is defined as
\begin{align*}
   E^+_\infty(s,\cD_{\ut}\otimes\sgn^k)&=\,\gamma_\infty(s,\sgn^k)^{-1}\prod_{j=1}^n e^{-(s+t_j-j)\frac{\pi i}{2}}\Gamma_{\bC}(s+t_j-j)\\
   (\text{resp. }E^-_\infty(s,\cD_{\ut}\otimes\sgn^k)&=\,\prod_{j=1}^ne^{-(s+t_j-j)\frac{\pi i}{2}} \Gamma_{\bC}(s+t_j-j)).
\end{align*}
(See also \cite[\S1]{SLF} \cite[\S2.3]{LRtz} for the formulas of the modified Euler factor at $p$ for $p$-adic interpolation in our case.)

Let $\sC$ (resp. $\sC_P$) be a geometrically irreducible component of the spectrum of the Hecke algebra acting on ordinary (resp. $P$-ordinary) families of Siegel modular forms of genus $n$ and tame principal level $N$ as in \cite{SLF} (resp. \cite{LRtz}). Here $P\subset\GL(n)$ is a standard parabolic subgroup corresponding to a partition $n=n_1+\dots+n_d$. Let $\phi:\bQ^\times\backslash\bA^\times\ra\bC^\times$ be Dirichlet character with conductor dividing $N\geq 3$. In \cite[Theorem 1.0.1]{SLF} (resp. \cite[Theorem 2.6.2]{LRtz}), we constructed an $(n+1)$-variable $p$-adic $L$-function  $\mu_{\sC,\phi,\beta_1,\beta_2}$ (resp. a $(d+1)$-variable $p$-adic $L$ function $\mu_{\sC_P,\phi,\beta_1,\beta_2}$) associated to $\sC$ and $\phi$ (resp. $\sC_{P}$ and $\phi$) for a given pair of positive definite symmetric $n\times n$ matrices $(\beta_1,\beta_2)$ with entries in $\bQ$.

\begin{coro}
If the weight projection is {\'e}tale at the classical point $x\in\sC(\ol{\bQ}_p)$ (resp. $x\in\sC_P(\ol{\bQ}_p)$) with image $\utau_x\in \Hom_{\cont}\left(T_n(\bZ_p),\ol{\bQ}^\times_p\right)$ (resp. $\utau^P_x\in\Hom_{\cont}\left(T_P(\bZ_p),\ol{\bQ}^\times_p\right)$), and if the algebraic part $k$ of $\kappa$ and the algebraic part $\ut_x$ of $\utau_x$ (resp. $\ut^P_x$ of $\utau^P_x$) satisfy $t_{x,1}\geq\dots\geq t_{x,n}\geq k\geq n+1$ (resp. $t^P_{x,1}\geq\dots\geq t^P_{x,d}\geq k\geq n+1$), then
\begin{align*}
   &\left(\int_{\bZ^\times_p}\kappa\,d\mu_{\sC,\phi,\beta_1,\beta_2}\right)(x)\\
   =&\,i^{-\frac{3n^2+n}{2}}2^{\frac{n^2-3n}{2}}\phi(-1)^n\,\vol\left(\wh{\Gamma}(N)\right) \frac{p^{n^2}(p-1)^n}{\prod_{j=1}^n(p^{2j}-1)}\cdot \frac{2^{-\sum_{j=1}^n t_{x,j}}}{\dim\left(\GL(n),\ut_x\right)}\sum_{\varphi\in\fs_x}\frac{\fc(\beta_1,\varphi)\fc(\beta_2,eW(\varphi))}{\left<\varphi,\ol{\varphi}\right>}\\
	 &\times E^+_p(k-n,\pi_x\times\phi^{-1}\chi^{-1})\,E^+_\infty(k-n,\pi_x\times\phi^{-1}\chi^{-1})\, L^{Np\infty}(k-n,\pi_x\times\phi^{-1}\chi^{-1}),
\end{align*}
and resp.
\begin{align*}
   &\left(\int_{\bZ^\times_p}\kappa\,d\mu_{\sC_P,\phi,\beta_1,\beta_2}\right)(x)\\
   =&\,i^{\frac{n^2-n}{2}}2^{\frac{3n^2+3n}{2}}\phi(-1)^n\mr{vol}\left(\wh{\Gamma}(N)\right)\frac{\prod_{l=1}^d\prod_{j=1}^{n_l}(1-p^{-j})}{\prod_{j=1}^n(1-p^{-2j})} \cdot\frac{2^{-\sum_{j=1}^n t_{x,j}}}{\dim\left(\GL(n),\ut_x\right)}\sum_{\varphi\in \fs_x}\frac{\fc(\beta_1,\varphi)\fc(\beta_2,e_P\pW(\varphi))}{\left<\varphi,\ol{\varphi}\right>}\\
   &\times E^-_p(n+1-k,\pi_x\times\phi\chi)\,E^-_\infty(n+1-k,\pi_x\times\phi\chi)\, L^{Np\infty}(n+1-k,\pi_x\times\phi\chi).
\end{align*}
Here the finite set $\fs_x$ consists of an orthogonal basis of ordinary (resp. $P$-ordinary) holomorphic Siegel modular forms of genus $n$, weight $\ut_x$ (resp. $(\underbrace{t^P_{x,1},\dots, t^P_{x,1}}_{n_1},\underbrace{t^P_{x,2},\dots, t^P_{x,2}}_{n_2},\dots,\underbrace{t^P_{x,d},\dots, t^P_{x,d}}_{n_d})$) and tame principal level $N$ on which the Hecke algebra acts via the Hecke eigensystem parameterized by $x$. For $\beta_i$, $i=1,2$, $\fc(\cdot,\beta_i)$ denotes the $\beta_i$-th coefficient of the $q$-expansion. See \cite[Theorem 1.0.1]{SLF} (resp. \cite[Theorem 2.6.2]{LRtz}) for the definition of the Siegel modular form $eW(\varphi)$ (resp. $e_P\pW(\varphi)$).
\end{coro}
Note that in the above interpolation formulas, the second line aligns with the conjecture of Coates--Perrin-Riou and the first line is independent of the cyclotomic variable $\kappa$. The interpolation formula for $\mu_{\sC,\phi,\beta_1,\beta_2}$ follows directly from Theorem~\ref{thm:Main} and \cite[Theorem 1.0.1]{SLF}, and the one for $\mu_{\sC,\phi,\beta_1,\beta_2}$ follows from Theorem~\ref{thm:Main}, \cite[Theorem 2.6.2]{LRtz} plus the functional equation for the archimedean doubling zeta integrals \cite{LapidRallis}\cite[(2.7.1)(2.7.2)]{LRtz}. 

We end the introduction by sketching the idea of our computation of $Z_\infty\left(f_{k,\ut},v^\vee_{\ut},v_{\ut}\right)$. The section $f_{k,\ut}$ is constructed by applying differential operators to the canonical holomorphic section of scalar weight $k$ (defined as in \eqref{eq:cansec}). The restriction of $f_{k,\ut}$ to $G(\bR)\times G(\bR)$ can be expressed as a matrix coefficient of the Weil representation (of $\Sp(2n,\bR)\times\OO(2k,\bR)$). We work with the Schr{\"o}dinger model of the Weil representation. Denote by $M_{a,b}$ the space of $a\times b$ matrices. There exists a polynomial $P^0_{k,\ut}\in\bC[M_{2n,2k}]$ such that $f_{k,\ut}\left(S^{-1}_H\iota(g,1)\right)$ equals the matrix coefficient attached to the Schwartz function $\phi_{P^0_{k,\ut}}(X)=P^0_{k,\ut}(X)e^{-\pi\Tr X\ltrans{X}}$ on $M_{2n,2k}(\bR)=M_{n,2k}(\bR)\times M_{n,2k}(\bR)$ (Proposition~\ref{prop:fMC}). Let $\fH_{n,2k}\subset\bC[M_{n,2k}]$ be the space of pluri-harmonic polynomials as defined in \cite{KVWeil}. It is acted on by $\GL(n)$ through the left transpose translation and by $\OO(2k)$ acts through the right translation. Inside $\fH_{n,2k}\otimes\fH_{n,2k}$, there is a unique irreducible component (as an algebraic $\GL(n)\times \OO(2k)$-representation) on which the action of $\GL(n)$ has highest weight $\ut$. The polynomial $P^0_{k,\ut}$ has a projection $P^{\hol,\inv}_{k,\ut}$ into $V_{\fH,k,\ut}\otimes V_{\fH,k,\ut}$ (which corresponds to the holomorphic projection on the restriction to $G\times G$ of nearly holomorphic Siegel Eisenstein series on $H$). The Schwartz function $\phi_{P^{\hol,\inv}_{k,\ut}}(X)= P^{\hol,\inv}_{k,\ut}(X)e^{-\pi\Tr X\ltrans{X}}$ belongs to the unique direct summand of the $G(\bR)\times G(\bR)$-representation $\cS\left(M_{n,2k}(\bR)\times M_{n,2k}(\bR)\right)^{\OO(2k,\bR)}$ which is isomorphic to $\cD_{\ut}\boxtimes\cD_{\ut}$.  Taking into account Harish-Chandra's formulas on formal degrees of discrete series, we reduce to compute the inner product associated to $\phi_{P^{\hol,\inv}_{k,\ut}}$ via \eqref{eq:Spairing}. 

In general, it is very difficult to write down explicitly the polynomial $P^{\hol,\inv}_{k,\ut}$, but it can be characterized uniquely (up to scalar) inside $V_{\fH,k,\ut}\otimes V_{\fH,k,\ut}$ by its invariance under the diagonal action of $\OO(2k)$ and being of highest weight for the $\GL(n)\times\GL(n)$-action. We introduce another two polynomials $Q_{k,\ut}\otimes\wt{Q}_{k,\ut}$ and $\cI_{k,\ut}$ in $V_{\fH,k,\ut}\otimes V_{\fH,k,\ut}$. They are characterized (up to scalar) by that $Q_{k,t}$ is of highest weight for the action of both $\GL(n)$ and $\OO(2k)$, $\wt{Q}_{k,\ut}$ is of highest weight for the action of $\GL(2n)$ and lowest weight for that of $\OO(2k)$, and $\cI_{k,\ut}$ is invariant under $(\ba,\ltrans{\ba}^{-1})$, $\ba\in\GL(n)$, and is of highest-lowest weight for the action of $\OO(2k)\times\OO(2k)$. The way we compute the inner product of $\phi_{P^{\hol,\inv}_{k,\ut}}$ is to link it first to that of $\phi_{Q_{k,\ut}}\otimes\phi_{\wt{Q}_{k,\ut}}$, and then to that of $\phi_{\cI_{k,\ut}}$. Both $Q_{k,\ut}\otimes\wt{Q}_{k,\ut}$ and $\cI_{k,\ut}$ are very easy to be written down explicitly and the inner product associated to $\phi_{\cI_{k,\ut}}$ is particularly easy to compute. \\

\noindent
\textbf{Notation.} 
For a positive integer $m$, define the algebraic groups $\Sp(2m)$ and $\OO(2m)$ as
\begin{align*}
   \Sp(2m)=&\left\{g\in\GL(2m):\ltrans{g}\begin{pmatrix}0&\bid_m\\-\bid_m&0\end{pmatrix}g=\begin{pmatrix}0&\bid_m\\-\bid_m&0\end{pmatrix}\right\},\\
   \OO(2m)=&\left\{\gamma\in\GL(2m):\ltrans{\gamma}\gamma=\bid_{2m}\right\}.
\end{align*}
For a pair of positive integers $m_1,m_2$, we denote by $M_{m_1,m_2}$ the space of $m_1\times m_2$ matrices. 

Let $G=\Sp(2n)_{/\bQ}$ and $H=\Sp(4n)_{/\bQ}$, and we fix the embedding
\begin{align*}
   \iota:G\times G&\lhra H\\
	   \left(\begin{array}{cc}a_1&b_1\\c_1&d_1\end{array}\right)\times \left(\begin{array}{cc}a_2&b_2\\c_2&d_2\end{array}\right)&\longmapsto \left(\begin{array}{cccc}a_1&0&b_1&0\\0&a_2&0&b_2\\c_1&0&d_1&0\\0&c_2&0&d_2\end{array}\right).
\end{align*}
The standard Siegel parabolic subgroup $Q_H\subset H$ consists of elements whose lower left $2n\times 2n$-blocks are $0$, and the doubling Siegel parabolic subgroup $Q^\diamond_H\subset H$ is defined as
\begin{align*}
   Q^\diamond_H&=S_HQ_HS^{-1}_H, &S_H=\begin{pmatrix}\bid_n&0&0&0\\0&\bid_n&0&0\\0&\bid_n&\bid_n&0\\\bid_n&0&0&\bid_n\end{pmatrix}.
\end{align*}
Denote by $K_G$ the maximal compact subgroup of $G(\bR)$ which we identify with $\U(n,\bR)=\{g\in\GL(n,\bC):\ltrans{\ol{g}}g=\bid_n\}$ via $\begin{pmatrix}a&b\\-b&a\end{pmatrix}\mapsto ai+b$.
We fix the Haar measure on $G(\bR)$ as the product measure where the one on $K_G$ has total volume $1$ and the one on $G(\bR)/K_G\cong\bH_n=\{z=x+iy\in M_{n,n}(\bC),\,\ltrans{z}=z,\,y>0\} $ is $\det(y)^{-n-1}\prod\limits_{1\leq i\leq j\leq n}dx_{ij}\,dy_{ij}$

\begin{Ac}
I would like to thank Paul Garrett, Michael Harris, Dongwen Liu, Binyong Sun and Eric Urban for many useful communications and insightful suggestions.
\end{Ac}

\section{Reducing to a simpler integral}

\subsection{The archimedean section in the doubling zeta integral}\label{sec:fkt}
We first recall the choice in \cite{SLF} of the section $f_{k,\ut}$ inside the degenerate principal series 
$$I(k)=\Ind^{H(\bR)}_{Q_H(\bR)}\sgn^k|\cdot|^{k-\frac{2n+1}{2}}=\left\{f:H(\bR)\ra\bC\text{ smooth} :\,f\left(\begin{pmatrix}A&B\\0&\ltrans{A}^{-1}\end{pmatrix}h\right)=(\det A)^k f(h)\right\}.$$
Inside $I(k)$, there is the canonical holomorphic section $f_{k,k}$ defined as
\begin{equation}\label{eq:cansec}
   f_{k,k}\left(\begin{pmatrix}A&B\\C&D\end{pmatrix}\right)=\det(Ci+D)^{-k}.
\end{equation}

The torus $\bC^\times$ acts on $H(\bR)$ by
\begin{align*}
   z\cdot h&=\begin{pmatrix}u\bid_{2n}&v\bid_{2n}\\-v\bid_{2n}&u\bid_{2n}\end{pmatrix}h\begin{pmatrix}u\bid_{2n}&v\bid_{2n}\\-v\bid_{2n}&u\bid_{2n}\end{pmatrix}^{-1},&z=u+iv\in\bC^\times.
\end{align*}
The induced action on $(\Lie H)_{\bC}$ decomposes it as
\begin{equation*}
   (\Lie H)_{\bC}=(\Lie H)^{-1,1}_{\bC}\oplus (\Lie H)^{0,0}_{\bC}\oplus (\Lie H)^{1,-1}_{\bC},
\end{equation*}
where $(\Lie H)^{a,b}_{\bC}$, $a,b=0,\pm 1$, is the subspace on which $z\in\bC^\times$ acts by the scalar $z^{-a}\ol{z}^{-b}$. Let $\fq^+_{H}=(\Lie H)^{-1,1}_{\bC}$ and we fix the following basis of it
\begin{align*}
   \wh{\mu}^+_{H,ij}=\wh{\mu}^+_{H,ji}&=\fc\begin{pmatrix}0&E_{ij}+E_{ji}\\0&0\end{pmatrix}\fc^{-1},&1\leq i,j\leq 2n,
\end{align*}
where $\fc=\frac{1}{\sqrt{2}}\begin{pmatrix}\bid_{2n}&i\bid_{2n}\\i\bid_{2n}&\bid_{2n}\end{pmatrix}$ and $E_{ij}$ is the $2n\times 2n$ matrix with $1$ in the $(i,j)$ entry and $0$ elsewhere. Define the following matrices with entries in $\fq^+_H$,
\begin{align*}
   \wh{\mu}^+_H=\begin{blockarray}{ccc}n&n&\\\begin{block}{(cc)c}\wh{\mu}^+_{H,1}&\wh{\mu}^+_{H,0}&n\\\ltrans{\wh{\mu}}^+_{H,0}&\wh{\mu}^+_{H,2}&n\\\end{block}\end{blockarray}=\left(\wh{\mu}^+_{ij}\right)_{1\leq i, j\leq 2n}.
\end{align*}
Explicit formulas on the equivalence between the action of $\fq^+_H$ and the Maass--Shimura differential operators are given in \cite[\S2.4]{SLF}. 

For $(k,\ut)$ with $t_1\geq\dots t_n\geq k$, define the following polynomial  on entries of $n\times n$ matrices,
$$
   \fQ_{k,\ut}=\prod_{j=1}^{n-1}\ddet^{t_j-t_{j+1}}_j\, \ddet^{k-t_n},
$$ 
where $\det_j$ stands for the determinant of the upper left $j\times j$ block. The section $f_{k,\ut}\in I(k)$ is defined as
\begin{equation}\label{eq:ourf}
   f_{k,\ut}=\fQ_{k,\ut}\left(\frac{\wh{\mu}^+_{H,0}}{4\pi i}\right)\cdot f_k.
\end{equation}

Our goal is to compute the following archimedean zeta integral which appears in the interpolation formulas of the $p$-adic $L$-functions for Siegel modular forms \cite[Theorem 1.0.1]{SLF},
\begin{equation}\label{eq:Zinfty}
   Z_\infty\left(f_{k,\ut},v^\vee_{\ut},v_{\ut}\right)=\int_{G(\bR)}f_{k,\ut}\left(S^{-1}_H\iota(g,1)\right)\left<g\cdot v^\vee_{\ut},v_{\ut}\right>\,dg,
\end{equation}
where $v_{\ut}$ is the highest weight vector inside the lowest $K_G$-type of the weight $\ut$ holomorphic discrete series $\cD_{\ut}$ of $G(\bR)$, and $v^\vee_{\ut}$ is the dual vector of $v_{\ut}$ inside the contragredient representation of $\cD_{\ut}$. The pairing $\left<\,,\,\right>$ appearing in the above identity is the natural pairing between $\cD_{\ut}$ and $\cD^\vee_{\ut}$ under which $\left<v^\vee_{\ut},v_{\ut}\right>=1$. The convergence and nonvanishing of the integral \eqref{eq:Zinfty} have been proved in \cite[Proposition 4.3.1]{SLF} using ideas from \cite{JV,HaBun,LiThetaCoh}.

\subsection{Basics on Weil representation}\label{sec:weil}
We recall some formulas on the Schr{\" o}dinger model of Weil representation which will be used in our computation of \eqref{eq:Zinfty}. Let $m$ be a positive integer and $\OO(2k,\bR)$ be the definite orthogonal group attached to the symmetric form $(x,y)=\sum_{j=1}^{2k} x_jy_j$ on $\bR^{2k}$. Denote by $\omega_{2k}$ the Weil representation of $\Sp(2m,\bR)\times \OO(2k,\bR)$ on $\cS(M_{m,2k}(\bR))$, the space of $\bC$-valued Schwartz functions on $M_{m,2k}(\bR)$. 

For $\phi\in\cS(M_{m,2k}(\bR))$ and $\gamma\in\OO(2k,\bR)$, $\begin{pmatrix}a&b\\0&\ltrans{a}^{-1}\end{pmatrix},\begin{pmatrix}0&\bid_m\\-\bid_m&0\end{pmatrix}\in\Sp(2m,\bR)$, we have
\begin{align*}
   \omega_{2k}(1,\gamma)\phi(x)&=\phi(x\gamma),\\
   \omega_{2k}\left(\begin{pmatrix}a&b\\0&\ltrans{a}^{-1}\end{pmatrix},1\right)\phi(x)&=(\det a)^k\,e^{\pi i\,\Tr b\ltrans{a}\, x\ltrans{x}}\phi(\ltrans{a}x),\\
   \omega_{2k}\left(\begin{pmatrix}0&\bid_m\\-\bid_m&0\end{pmatrix},1\right)\phi(x)&=i^{-mk}\int_{M_{m,2k}(\bR)}\phi(y)\,e^{2\pi i\,\Tr x\ltrans{y}}\,dy.
\end{align*}
More generally, for $g=\begin{pmatrix}a&b\\c&d\end{pmatrix}\in\Sp(2m,\bR)$,
\begin{equation}\label{eq:weil}
   \omega_{2k}\left(\begin{pmatrix}a&b\\c&d\end{pmatrix},1\right)\phi(x)=\epsilon(g)^k i^{-k\,\mr{rank}\, c}\int_{(\bR^m/\ker c)\otimes \bR^{2k}}\phi(\ltrans{c}y+\ltrans{a}x)\,e^{2\pi i\,\left(\frac{1}{2}\Tr c\ltrans{d}\,y\ltrans{y}+\Tr\,c\ltrans{b}\,x\ltrans{y}+\frac{1}{2}\Tr a\ltrans{b}\,x\ltrans{x}\right)}\,dy,
\end{equation}
where $\epsilon(g)=\sgn(\det a_1a_2)$ for 
\begin{align*}
   g&=\begin{pmatrix}a_1&b_1\\0&\ltrans{a}^{-1}_1\end{pmatrix}\begin{pmatrix}\bid_{m-j}&0&0&0\\0&0&0&\bid_j\\0&0&\bid_{m-j}&0\\0&-\bid_j&0&0\end{pmatrix}\begin{pmatrix}a_2&b_2\\0&\ltrans{a}^{-2}_1\end{pmatrix}, &j=\mr{rank}\,c.
\end{align*}

\subsection{Siegel--Weil sections and matrix coefficients}
The Siegel--Weil sections inside $I(k)$ are images of the map
\begin{align*}
   f_{\mr{SW}}:\cS(M_{2n,2k}(\bR))&\lra I(k)\\
    \phi&\longmapsto f_{\mr{SW}}(h)= \omega_{2k}(h,1)\phi(0).
\end{align*}
The canonical holomorphic section \eqref{eq:cansec} is the image of the Gaussian function, {\it i.e.}
$$
   f_{k,k}=f_{\mr{SW}}\left(X\mapsto e^{-\pi\,\Tr X\ltrans{X}}\right).
$$

Multiplying with the Gaussian function defines an embedding from the space of polynomial functions into the space of Schwartz functions. For a positive integer $m$, we define
\begin{equation}\label{eq:mPS}
\begin{aligned}
   \bC[M_{m,2k}]&\lhra \cS(M_{m,2k}(\bR))\\
   P(x)&\longmapsto \phi_P(x)=P(x)\, e^{-\pi\,\Tr x\ltrans{x}}.
\end{aligned}
\end{equation}
Later we will mainly focus on the image of \eqref{eq:mPS} for $m=n,2n$ which is dense inside $\cS(M_{m,2k}(\bR))$ and stable under the action of $(\Lie \Sp(2m))_{\bC}\times\OO(2k,\bR)$. In particular, our chosen section $f_{k,\ut}$ in \eqref{eq:ourf} belongs to the image of the composition of $f_{\mr{SW}}$ and \eqref{eq:mPS} with $m=2n$. 

Given $\sigma\in\Sym(2n,\bR)$, set
$$
   \wh{\mu}^+_{H,\sigma}=\fc\begin{pmatrix}0&\sigma\\0&0\end{pmatrix}\fc^{-1}.
$$
If $P(X)$ is a homogeneous polynomial on $M_{2n,2k}$ of degree $d$, an easy computation using the formulas in \S\ref{sec:weil} gives 
\begin{equation}\label{eq:mu-sigma}
   \omega_{2k}\left(\wh{\mu}^+_{H,\sigma},1\right) \left(\phi_P(X)\right)=2\pi i\,\left(\Tr\sigma X\ltrans{X}\right)\Big(P(X)+\text{polymonial in $X$ of degree $<d$}\Big)e^{-\pi\,\Tr X\ltrans{X}},
\end{equation}
from which it follows:

\begin{prop}\label{prop:p0}
There exists a polynomial function $P^0_{k,\ut}$ on $M_{2n,2k}(\bR)$ of the form
$$
   P^0_{k,\ut}\left(X=\begin{pmatrix}x_1\\x_2\end{pmatrix}\right)=\fQ_{k,\ut}(x_1\ltrans{x}_2)+\textnormal{polynomial in $X$ of degree $< 2\sum t_j-2nk$}
$$
such that
$$
  f_{k,\ut}=2^{-\sum t_j+nk}\,f_{\mr{SW}}\left(\phi_{P^0_{k,\ut}}\right).
$$
\end{prop}

Let $\cS(M_{n,2k}(\bR))^\mvw$ be the $G(\bR)\times\OO(2k,\bR)$-representation which equals $\cS(M_{n,2k}(\bR))$ as a $\bC$-vector space and on which $(g,\gamma)\in G(\bR)\times\OO(2k,\bR)$ acts through $\omega_{2k}(g^\mvw,\gamma)$. Here the superscript $^\mvw$ denotes the conjugation by $\begin{pmatrix}0&\bid_n\\\bid_n&0\end{pmatrix}$ (which is also called the MVW involution on $G$ and is introduced in \cite{MVW}). 

As $G(\bR)\times\OO(2k,\bR)$-representations, $\cS(M_{n,2k}(\bR))^\mvw$ is isomorphic to the contragredient representation of $\cS(M_{n,2k}(\bR))$, and we have the following $G(\bR)\times\OO(2k,\bR)$-invariant $\bC$-linear pairing
\begin{equation}\label{eq:Spairing}
\begin{aligned}
   \left<\,,\,\right>:\cS(M_{n,2k}(\bR))\times \cS(M_{n,2k}(\bR))^\mvw&\lra \bC\\
   (\phi_1,\phi_2)&\longmapsto \left<\phi_1,\phi_2\right>=\int_{M_{n,2k}(\bR)}\phi_1(x)\,\omega_{2k}\left(\begin{pmatrix}0&-\bid_n\\\bid_n&0\end{pmatrix},1\right)\phi_2(x)\,dx.
\end{aligned}
\end{equation}

Suppose that $\phi\in\cS(M_{2n,2k}(\bR))$ has the form $\phi\begin{pmatrix}x_1\\x_2\end{pmatrix}=\phi_1(x_1)\,\phi_2(x_2)$ with $\phi_1,\phi_2\in\cS(M_{n,2k}(\bR))$. Direct computation using \eqref{eq:weil} shows that
\begin{equation}\label{eq:resmc}
   f_{\mr{SW}}(\phi)\left(S^{-1}_H\iota(g,1)\right)=i^{nk}\left<\omega_{2k}(g,1)\phi_1,\phi_2\right>.
\end{equation}   
(The $G(\bR)$-invariance of the pairing $\left<\,,\,\right>$ is also reflected by the invariance property  
$$f_{\mr{SW}}\left(S^{-1}_H\iota(g_1,g_2)\right)=f_{\mr{SW}}\left(S^{-1}_H\iota (g g_1,g^\mvw g_2)\right)$$
of Siegel--Weil sections.)

For each $g\in G(\bR)$, we can define a linear functional $\MC_{\omega,2k}(g,\cdot)$ on $\cS(M_{2n,2k}(\bR))$ by taking the matrix coefficients of the Weil representation on $\cS(M_{n,2k}(\bR))\times \cS(M_{n,2k}(\bR))^\mvw$. More precisely, we define it as
\begin{align*}
   \MC_{\omega,2k}(g,\cdot):\cS(M_{2n,2k}(\bR))&\lra \bC\\
   \phi&\longmapsto \MC_{\omega,2k}(g,\phi)=\int_{M_{n,2k}(\bR)}\omega_{2k}\left(\iota\left(1,\begin{psm}0&-\bid_n\\\bid_n&0\end{psm}\right),1\right)\phi\begin{pmatrix}x\\x\end{pmatrix}\,dx.
\end{align*}
Combining Proposition~\ref{prop:p0} and \eqref{eq:resmc} we immediately get
\begin{prop}\label{prop:fMC}
\begin{equation*}
   f_{k,\ut}\left(S^{-1}_H\iota(g,1)\right)=i^{nk}\,2^{-\sum t_j+nk}\,\MC_{\omega,2k}\left(g,\phi_{P^0_{k,\ut}}\right).
\end{equation*}
\end{prop}

\subsection{Pluri-harmonic polynomials}
In \cite{KVWeil}, pluri-harmonic polynomials are introduced to study the theta correspondence between $\Sp(2m,\bR)$ and $\OO(2k,\bR)$. The image of pluri-harmonic polynomials by the map \eqref{eq:mPS} corresponds to lowest $K$-types in the holomorphic discrete series appearing in the decomposition of the Weil representation. 

\begin{defn}
A polynomial $P(x)$ in $\bC[M_{n,2k}]$ is called pluri-harmonic if 
$$
   \Delta_{ij}P(x)=\sum_{r=1}^{2k}\frac{\partial^2}{\partial x_{ir}\partial x_{jr}}P(x)=0.
$$
for all $1\leq i,j\leq n$. We denote the space of pluri-harmonic polynomials on $M_{n,2k}$ as $\fH_{n,2k}$.
\end{defn}

For
\begin{align*}
   x&=\begin{blockarray}{ccc}k&k&\\\begin{block}{(cc)c}u&v&n\\\end{block}\end{blockarray}\in M_{n,2k}(\bR),
\end{align*}
define
\begin{align*}
   \bz(x)&=u+iv\in M_{n,k}(\bC),&\ol{\bz}(x)&=u-iv\in M_{n,k}(\bC).
\end{align*}
Given $t_1\geq\dots\geq t_n\geq k$, define $Q_{k,\ut}$, $\wt{Q}_{k,\ut}\in\bC[M_{n,2k}]$ as
\begin{align*}
   Q_{k,\ut}(x)&=\fQ_{k,\ut}\left((\bz(x)_{ij})_{1\leq i,j\leq n}\right), &\wt{Q}_{k,\ut}(x)&=\fQ_{k,\ut}\left((\ol{\bz}(x)_{ij})_{1\leq i,j\leq n}\right),
\end{align*}
which are easily seen both pluri-harmonic.

The group $\GL(n,\bC)$ (resp. $\OO(2k,\bC)$) acts on $\bC[M_{n,2k}]$ by left transpose translation (resp. right translation). The pluri-harmonic polynomial $Q_{k,\ut}$ generates an irreducible algebraic $\GL(n,\bC)\times \OO(2k,\bC)$-representation $V_{\fH,k,\ut}$  inside $\fH_{n,2k}$, and we have 
\begin{equation*}%\label{eq:Vhkt}
   V_{\fH,k,\ut}\cong(\ut-k)\boxtimes \lambda(\ut,k),
\end{equation*}
where $\ut-k$ stands for the irreducible algebraic representation of $\GL(n)$ of highest weight $(t_1-k,t_2-k,\dots,t_n-k)$. The restriction to $\SO(2k,\bC)$ of the irreducible algebraic $\OO(2k,\bC)$-representation  $\lambda(\ut,k)$ is still irreducible, and its highest weight is
$$
   (t_1-k,t_2-k,\dots,t_n-k,0,\dots,0).
$$   
The pluri-harmonic polynomial $\wt{Q}_{k,\ut}$ also generates $V_{\fH,k,\ut}$. Both $Q_{k,\ut}$ and $\wt{Q}_{k,\ut}$ are of the highest weight for the action of $\GL(n,\bC)$, and $Q_{k,\ut}$ (resp. $\wt{Q}_{k,\ut}$) is of the highest (resp. lowest) weight for the action of $\OO(2k,\bC)$.

For pluri-harmonic polynomials on $M_{n,2k}$, the $\GL(n,\bC)$-action is closely related to the action of $K_G\cong \U(n,\bR)$ via the Weil representation on their images under the map \eqref{eq:mPS}.
\begin{prop}\label{prop:phequiv}
For $P\in\fH_{n,2k}$ and $\begin{pmatrix}a&b\\c&d\end{pmatrix}\in G(\bR)$,
$$
   \omega_{2k}\left(\begin{pmatrix}a&b\\c&d\end{pmatrix},1\right)\phi_P(x)=\det(ci+d)^{-k} P\left((ci+d)^{-1}x\right)\,e^{i\pi\,\Tr(ai+b)(ci+d)^{-1}x\ltrans{x}}.
$$
In particular, for $\begin{pmatrix}a&b\\-b&a\end{pmatrix}\in K_G$ ($ai+b\in \U(n,\bR)$) we have
$$
   \omega_{2k}\left(\begin{pmatrix}a&b\\-b&a\end{pmatrix},1\right)\phi_P=\det(ai+b)^k\phi_{(ai+b)\cdot P}.
$$
\end{prop}
\begin{proof}
It follows from combining the formula \eqref{eq:weil} and \cite[Lemma 4.5]{KVWeil}.
\end{proof}

By studying the space of pluri-harmonic polynomials, the following theorem on theta correspondence between $G$ and $\OO(2k,\bR)$ is deduced in \cite{KVWeil}.
\begin{thm}[{\cite[Proposition 6.11, Corollary 6.12, Theorem 6.13]{KVWeil}}]
Assume that $k\geq n+1$.
\begin{itemize}
\item (Theta correspondence) As a representation of $(\Lie G)_{\bC}\times \OO(2k,\bR)$,
\begin{equation}\label{eq:ThetaCor}
   \mr{Im}\,\eqref{eq:mPS}_{m=n}\cong\bigoplus_{\substack{\ut\,\mr{dominant}\\t_n\geq k}}\cD_{\ut}\boxtimes \lambda(\ut,k).
\end{equation}
\item Under the isomorphism \eqref{eq:ThetaCor}, the lowest $K_G$-types of the holomorphic discrete series correspond to pluri-harmonic polynomials. In particular, 
$$
   \phi_{Q_{k,\ut}}\in\left(\textnormal{lowest $K_G$-type in }\cD_{\ut}\right)\boxtimes\lambda(\ut,k),
$$
and is the highest weight vector for the action of $K_G\times\OO(2k,\bR)$.
\end{itemize}
\end{thm}

We record here another fact about pluri-harmonic polynomials which will be used in the next section. Denote by $\bC[x\ltrans{x}]$ the subspace of $\bC[M_{n,2k}]$ spanned by polynomials in entries of $x\ltrans{x}$. By \cite[Remark under Lemma 5.3]{KVWeil}, if $k\geq n$ then
\begin{equation*}
   \bC[M_{n,2k}]=\fH_{n,2k}\otimes_{\bC}\bC[x\ltrans{x}].
\end{equation*}
Hence for $t_1\geq\dots t_n\geq k\geq n+1$, we have
\begin{equation}\label{eq:otherK}
   \text{embedding $\eqref{eq:mPS}_{m=n}$}\,\Big( V_{\fH,k,\ut}\otimes_{\bC}\bC[x\ltrans{x}]\Big)\cong\cD_{\ut}\boxtimes\lambda(\ut,k).
\end{equation}

Next we consider the space $\fH_{n,2k}\otimes_{\bC}\fH_{n,2k}\subset\bC[M_{2n,2k}]$. First we define
\begin{align*}
   P^{\hol}_{k,\ut}&\in V_{\fH,k,\ut}\otimes_{\bC} V_{\fH,k,\ut}, &P^\hol_{k,\ut}\left(X=\begin{pmatrix}x_1\\x_2\end{pmatrix}\right)&=Q_{k,\ut}(x_1)\,\wt{Q}_{k,\ut}(x_2).
\end{align*}
(Note that $P^\hol_{k,\ut}$ does not belong to $\fH_{2n,2k}$.) The group $\OO(2k,\bC)$ acts on $\fH_{n,2k}\otimes_{\bC}\fH_{n,2k}$ diagonally by right translation. Since $\dim_{\bC}\left(\lambda(\ut,k)\otimes\lambda(\ut,k)\right)^{\OO(2k,\bC)}=1$, there exists a unique a pluri-harmonic polynomial
$$
   P^{\hol,\inv}_{k,\ut}\in \left(V_{\fH,k,\ut}\otimes_{\bC} V_{\fH,k,\ut}\right)^{\OO(2k,\bC)}
$$
such that it is the the highest weight vector for the action of $\GL(n,\bC)\times\GL(n,\bC)$ and its evaluation at
\begin{equation}\label{eq:ev}
   \begin{blockarray}{ccccc}n&k-n&n&k-n&\\\begin{block}{(cccc)c}\bid_n&0&0&0&n\\0&0&\bid_n&0&n\\\end{block}\end{blockarray}\,\begin{pmatrix}\bid_k&\bid_k\\i\bid_k&-i\bid_k\end{pmatrix}^{-1}
\end{equation}
equals $1$.

\begin{rmk}
The invariant pluri-harmonic polynomial $P^{\hol,\inv}_{k,\ut}$ has been used in \cite{Ib} to construct holomorphic differential operators. From $P^{\hol,\inv}_{k,\ut}$ one can define a polynomial $R_{k,\ut}$ on $\Sym(2n)$ such that $P^{\hol,\inv}_{k,\ut}(X)=R_{k,\ut}(X\ltrans{X})$. Suppose that $Z=(Z_{ij})_{1\leq i,j\leq 2n}$ is the coordinate of the Siegel upper half space for $\Sp(4n)$. Then $R_{k,\ut}\left(\frac{\partial}{\partial Z_{ij}}\right)$ composed with the restriction from $H$ to $G\times G$ sends holomorphic Siegel modular forms of scalar weight $k$ on $H$ to holomorphic Siegel modular forms of weight $\ut\boxtimes\ut$ on $G\times G$.
\end{rmk}

\subsection{Rewriting the zeta integral in terms of $Q_{k,\ut}$}
\begin{prop}\label{prop:ZW}
\begin{align*}
   Z_\infty\left(f_{k,\ut},v^\vee_{\ut},v_{\ut}\right)=i^{nk}2^{-\sum t_j+nk}\,\dim\lambda(\ut,k) \int_{G(\bR)}\left<\omega_{2k}(g,1)\phi_{Q_{k,\ut}},\phi_{\wt{Q}_{k,\ut}}\right>\left<g\cdot v^\vee_{\ut},v_{\ut}\right>\,dg.
\end{align*}
\end{prop}
\begin{proof}
The proof goes in two steps.

{\em Step 1}: Verify for all $g\in G(\bR)$ that
\begin{equation}\label{eq:ZWS1}
   \MC_{\omega,2k}\Big(g,\phi_{P^{\hol,\inv}_{k,\ut}}\Big)=\dim\lambda(\ut,k)\,\MC_{\omega,2k}\left(g,\phi_{P^\hol_{k,\ut}}\right)=\dim\lambda(\ut,k)\,\left<\omega_{2k}(g,1)\phi_{Q_{k,\ut}},\phi_{\wt{Q}_{k,\ut}}\right>.
\end{equation}
Let $N_n\subset \GL(n)$ be the unipotent radical of the standard Borel subgroup of upper triangular matrices. Then
\begin{align*}
   P^\hol_{k,\ut},\,P^{\hol,\inv}_{k,\ut}\in \left(V_{\fH,k,\ut}\otimes_{\bC} V_{\fH,k,\ut}\right)^{N_n(\bC)\times N_n(\bC)}.
\end{align*}
We fix a basis of $v_1,\dots,v_d$ of $\lambda(\ut,k)$ ($d=\dim\lambda(\ut,k)$) such that each $v_j$ is an eigenvector of the standard maximal torus consisting of 
\begin{align*}
   &\begin{psm}\bid_k&\bid_k\\i\bid_k&-i\bid_k\end{psm}\begin{psm}a_1&&&&&\\&\ddots&&&&\\&&a_k&&&\\&&&a^{-1}_1&&\\&&&&\ddots&\\&&&&&a^{-1}_k\end{psm}\begin{psm}\bid_k&\bid_k\\i\bid_k&-i\bid_k\end{psm}^{-1}\in\OO(2k,\bC),&a_1,\dots,a_k\in\bC^\times,
\end{align*}
and $v_1$ is the highest weight vector. Denote by $v^\vee_1,\dots,v^\vee_d$ the dual basis for $\lambda(\ut,k)^\vee$. We can fix an isomorphism
\begin{equation}\label{eq:Vlambda}
   \left(V_{\fH,k,\ut}\otimes_{\bC} V_{\fH,k,\ut}\right)^{N_n(\bC)\times N_n(\bC)}\stackrel{\sim}{\lra}\lambda(\ut,k)\otimes\lambda(\ut,k)^\vee
\end{equation}
such that $P^{\hol}_{k,\ut}$ is mapped to $v_1\otimes v^\vee_1$. Then this isomorphism maps $P^{\hol,\inv}_{k,\ut}$ to $C\,\sum_{j=1}^d v_j\otimes v^\vee_j$ for a constant $C$. We have the commutative diagram
$$
\xymatrix{
   \left(V_{\fH,k,\ut}\otimes_{\bC} V_{\fH,k,\ut}\right)^{N_n(\bC)\times N_n(\bC)}\ar[r]^-{\sim}\ar[dr]_{\txt{evaluated at \eqref{eq:ev}}\hspace{3em} }&\lambda(\ut,k)\otimes\lambda(\ut,k)^\vee\ar[d]^{\begin{array}{l}v_1\otimes v^\vee_1\mapsto 1\\v_i\otimes v_j\mapsto 0 \text{ if }i\neq 1\text{ or }j\neq 1\end{array}}\\&\bC
}.
$$
The composition of the horizontal and vertical maps sends $P^{\hol,\inv}_{k,\ut}$ to $C$, and the diagonal map sends it to $1$. Therefore, $C=1$ and \eqref{eq:Vlambda} maps $P^{\hol,\inv}_{k,\ut}$ to $\sum_{j=1}^d v_j\otimes v^\vee_j$. For each $g\in G(\bR)$, 
\begin{equation*}
   P_1\otimes P_2\longmapsto \MC_{\omega,2k}\left(g,\phi_{P_1}\otimes\phi_{P_2}\right)
\end{equation*}
defines an $\OO(2k,\bR)$-invariant pairing on $\left(V_{\fH,k,\ut}\otimes_{\bC} V_{\fH,k,\ut}\right)^{N_n(\bC)\times N_n(\bC)}$ and it must agree up to scalar with the composition of \eqref{eq:Vlambda} and the natural pairing on $\lambda(\ut,k)\otimes\lambda(\ut,k)^\vee$, which maps $v_1\otimes v^\vee_1$ to $1$ and $\sum_{j=1}^d v_j\otimes v^\vee_j$ to $\dim\lambda(\ut,k)$. Thus,
$$
   \MC_{\omega,2k}\Big(g,\phi_{P^{\hol,\inv}_{k,\ut}}\Big)=\dim\lambda(\ut,k)\,\MC_{\omega,2k}\left(g,\phi_{P^\hol_{k,\ut}}\right)
$$
and {\em Step 1} is done.

{\em Step 2}: Verify that
\begin{equation*}
   \int_{G(\bR)}\MC_{\omega,2k}\left(g,\phi_{P^0_{k,\ut}}\right)\left<g\cdot v^\vee_{\ut},v_{\ut}\right>\,dg=\int_{G(\bR)}\MC_{\omega,2k}\Big(g,\phi_{P^{\hol,\inv}_{k,\ut}}\Big)\left<g\cdot v^\vee_{\ut},v_{\ut}\right>\,dg.
\end{equation*}
By \eqref{eq:ThetaCor}, viewing $\cS(M_{2n,2k})^{\OO(2k,\bR)}$ as a representation of $G(\bR)\times G(\bR)$, it decomposes as a direct sum of $\cD_{\ut}\boxtimes\cD_{\ut}$ for $\ut$ dominant and $t_n\geq k$. We claim that
\begin{equation}\label{eq:diff}
   \phi_{P^0_{k,\ut}}-\phi_{P^{\hol,\inv}_{k,\ut}}\in  \bigoplus_{\substack{\ut'\text{ dominant}\\t'_n\geq k,\, \sum t'_j<\sum t_j}}\cD_{\ut'}\boxtimes\cD_{\ut'}.
\end{equation}
The claim implies that the matrix coefficient $\MC_{\omega,2k}\Big(\cdot,\phi_{P^0_{k,\ut}}-\phi_{P^{\hol,\inv}_{k,\ut}}\Big)$ pairs trivially with the matrix coefficient $\left<g\cdot v^\vee_{\ut},v_{\ut}\right>$. We are left to prove the claim. For $r\geq 0$, by considering the decomposition of $\bC[M_{n,2k}]^{\leq r}$ via the action of $\OO(2k,\bC)$, we know that
\begin{equation}\label{eq:hpr}
   \bC[M_{n,2k}]^{\leq r}\subset \fH^{r}_{n,2k}\oplus \fH^{\leq r-1}_{n,2k}\otimes\bC[x\ltrans{x}].
\end{equation}
Here we use the superscript $^{\leq r}$ (resp. $^{r}$) to indicate degree $\leq r$ (resp. homogeneous of degree $r$). The polynomial 
$$P'_{k,\ut}\left(X=\begin{pmatrix}x_1\\x_2\end{pmatrix}\right)= \fQ_{k,\ut}(x_1\ltrans{x}_2)$$
belongs to $\bC[M_{2n,2k}]^{\OO(2k,\bC)}$, and for the action of $\GL(n,\bC)\times\GL(n,\bC)$ it is fixed by $N_n(\bC)\times N_n(\bC)$ and of weight $(\ut,\ut)$. Hence its projection to the first factor on the right hand side of \eqref{eq:hpr} with $r=\sum t_j-nk$  must be a multiple of $P^{\hol,\inv}_{k,\ut}$. We have
\begin{equation}\label{eq:CR}
\begin{aligned}
  P'_{k,\ut}&=C\cdot P^{\hol,\inv}_{k,\ut}+R, &C\in\bC,\,R\in\bC[M_{2n,2k}]^{2\sum t_j-2nk}\cap (x_1\ltrans{x}_1,x_2\ltrans{x}_2)\bC[M_{2n,2k}]
\end{aligned}
\end{equation}
where $(x_1\ltrans{x}_1,x_2\ltrans{x}_2)\bC[M_{2n,2k}]$ denotes the ideal generated by entries of $x_1\ltrans{x}_1$ and $x_2\ltrans{x}_2$. Evaluating \eqref{eq:CR} at \eqref{eq:ev}, we get $C=1$. Taking into account Proposition~\ref{prop:p0}, we deduce that
\begin{equation*}
   P^0_{k,\ut}-P^{\hol,\inv}_{k,\ut}\in \left(\fH_{n,2k}\otimes\fH_{n,2k}\right)^{\leq 2\sum t_j-2nk-1}\otimes\bC[x_1\ltrans{x}_1,x_2\ltrans{x}_2]
\end{equation*}
which implies the claim \eqref{eq:diff}.

Combining {\em Step 1}, {\em Step 2} and Proposition~\ref{prop:fMC} proves the proposition. 
\end{proof}

\section{Computing the simpler integral}
In this section we compute the right hand side on the identity in Proposition~\ref{prop:fMC}.
\subsection{The formal degrees of $\cD_{\ut}$}
Let $V^{\OO\text{-HW}}_{\fH,k,\ut}$ (resp. $V^{\OO\text{-LW}}_{\fH,k,\ut}$) be the subspace of $V_{\fH,k,\ut}$ which is of the highest (resp. lowest) weight for the action of $\OO(2k,\bC)$. Thanks to \eqref{eq:otherK}, there exist isomorphisms
\begin{align*}
   \iota:&\,\text{embedding $\eqref{eq:mPS}_{m=n}$}\,\Big(V^{\OO\text{-HW}}_{\fH,k,\ut}\otimes_{\bC}\bC[x\ltrans{x}]\Big)\stackrel{\sim}{\lra} \cD_{\ut}, \\
   \iota^\vee:&\,\text{embedding $\eqref{eq:mPS}_{m=n}$}\,\Big(V^{\OO\text{-LW}}_{\fH,k,\ut}\otimes_{\bC}\bC[x\ltrans{x}]\Big)^\mvw\stackrel{\sim}{\lra} \cD^\vee_{\ut},
\end{align*}
and $Q_{k,\ut}\in V^{\OO\text{-HW}}_{\fH,k,\ut}$ (resp. $\wt{Q}_{k,\ut}\in V^{\OO\text{-LW}}_{\fH,k,\ut}$) is mapped to a multiple of $v_{\ut}$ (resp. $v^\vee_{\ut}$). Since the pairing \eqref{eq:Spairing} is $G(\bR)$-invariant, we know that for all $g\in G(\bR)$,
\begin{equation}\label{eq:SD}
   \frac{\left<\omega_{2k}(g,1)\phi_{Q_{k,\ut}},\phi_{\wt{Q}_{k,\ut}}\right>}{\left<\phi_{Q_{k,\ut}},\phi_{\wt{Q}_{k,\ut}}\right>}=\frac{\left<g\cdot v_{\ut},v^\vee_{\ut}\right>}{\left<v_{\ut},v^\vee_{\ut}\right>}.
\end{equation}
Therefore,
\begin{equation}\label{eq:WFD}
\begin{aligned}
   \int_{G(\bR)}\left<\omega_{2k}(g,1)\phi_{Q_{k,\ut}},\phi_{\wt{Q}_{k,\ut}}\right>\left<g\cdot v^\vee_{\ut},v_{\ut}\right>\,dg=&\left<\phi_{Q_{k,\ut}},\phi_{\wt{Q}_{k,\ut}}\right> \int_{G(\bR)}\left<g\cdot v_{\ut},v^\vee_{\ut}\right>\left<g\cdot v^\vee_{\ut},v_{\ut}\right>\,dg\\
   =&\,d(\cD_{\ut},dg)^{-1}\left<\phi_{Q_{k,\ut}},\phi_{\wt{Q}_{k,\ut}}\right>,
\end{aligned}
\end{equation}
where $d(\cD_{\ut},dg)$ is the formal degree of the holomorphic discrete series $\cD_{\ut}$ with respect to our chosen Haar measure $dg$ on $G(\bR)$.  

\begin{rmk}
Thanks to the compactness of the orthogonal group we consider here, the Weil representation decomposes as a direct sum of irreducible representations and we easily deduce the identity \eqref{eq:SD}. If the orthogonal group is not compact, then the ratio of the two sides will be a constant depending on $k,\ut$. Such a ratio has been computed in \cite{LiuD-U21,LiuD-Un1} for the unitary case when one of the unitary groups has signature $(n,1)$. 
\end{rmk}

Now it remains to compute the pairing $\left<\phi_{Q_{k,\ut}},\phi_{\wt{Q}_{k,\ut}}\right>$.

\subsection{The pairing of $\phi_{Q_{k,\ut}}$ with $\phi_{\wt{Q}_{k,\ut}}$}
First we define another element $\cI_{k,\ut}$ inside the space $V^{\OO\text{-HW}}_{\fH,k,\ut}\otimes V^{\OO\text{-LW}}_{\fH,k,\ut}$ as
\begin{align*}
   \cI_{k,\ut}(x_1,x_2)&=\fQ_{k,\ut}\left(\left(\ltrans{\bz}(x_1)\,\ol{\bz}(x_2)\right)_{\text{upper left $n\times n$-block}}\right), &(x_1,x_2)\in M_{n,2k}(\bR)\times M_{n,2k}(\bR).
\end{align*}
In order to see that $\cI_{k,\ut}\in V^{\OO\text{-HW}}_{\fH,k,\ut}\otimes V^{\OO\text{-LW}}_{\fH,k,\ut}$, we notice that the right translation on $\bz(x_1)$ and $\ol{\bz}(x_2)$ by
$$
   \begin{pmatrix}a_1&\ast&\cdots&\ast\\&a_2&\cdots&\ast\\&&\ddots&\vdots\\&&&a_k\end{pmatrix}\in\GL(k,\bC)
$$ 
is by $\prod_{j=1}^n a^{t_j-k}_j$ like $Q_{k,\ut}$, $\wt{Q}_{k,\ut}$. It is also easy to see that for all $\ba\in\GL(n,\bC)$,
\begin{equation}\label{eq:Iinv}
   \cI_{k,\ut}(\ba x_1,\ltrans{\ba}^{-1}x_2)=\cI_{k,\ut}(x_1,x_2).
\end{equation}

\begin{prop}\label{prop:WI}
$$
   \left<\phi_{Q_{k,\ut}},\phi_{\wt{Q}_{k,\ut}}\right>=(-i)^{\sum_{j=1}^n t_j}\dim\left(\GL(n),\ut\right)^{-1}\,\int_{M_{n,2k}(\bR)}\cI_{k,\ut}(x,x)\,e^{-2\pi\,\Tr x\ltrans{x}}\,dx,
$$
where $\dim\left(\GL(n),\ut\right)$ is the dimension of the irreducible algebraic $\GL(n)$-representation of highest weight $\ut$.
\end{prop}
\begin{proof}
The proof goes similarly as the first step in the proof of Proposition~\ref{prop:ZW}. Consider the $\bC$-linear pairing
\begin{equation}\label{eq:pairgl}
\begin{aligned}
   (\,,\,):V^{\OO\text{-HW}}_{\fH,k,\ut}\times V^{\OO\text{-LW}}_{\fH,k,\ut}&\lra \bC\\
   (Q_1,Q_2)&\longmapsto \left<\phi_{Q_1},\phi_{Q_2}\right>.
\end{aligned}
\end{equation}
For $\ba\in\GL(n,\bC)$, denote by $L(\ba)$ the left translation by $\ba$ on $\bC[M_{n,2k}]$. Due to the $G(\bR)$-invariance of the pairing \eqref{eq:Spairing} and Proposition~\ref{prop:phequiv}, the pairing $(\,,\,)$ has the equivariance property that
\begin{align*}
   \Big(L(a+ib)Q_1,L(ai-b)Q_2\Big)=\bigg<\omega_{2k}\left(\begin{pmatrix}a&b\\-b&a\end{pmatrix},1\right)\phi_{Q_1},\omega_{2k}\bigg(\begin{pmatrix}a&b\\-b&a\end{pmatrix}^\mvw,1\bigg)\phi_{Q_2}\bigg>=(Q_1,Q_2)
\end{align*}
for all $a+ib\in \U(n,\bR)$. On the other hand, for the $\GL(n,\bC)$-action by left transpose translation, $V^{\OO\text{-HW}}_{\fH,k,\ut}$ and $V^{\OO\text{-LW}}_{\fH,k,\ut}$ are both isomorphic to the irreducible algebraic $\GL(n)$-representation of highest weight $\ut$. Thus, there exists a basis $v_1,\dots,v_d$ (resp. $w_1,\dots,w_d$) of $V^{\OO\text{-HW}}_{\fH,k,\ut}$ (resp. $V^{\OO\text{-LW}}_{\fH,k,\ut}$) ($d=\dim\left(\GL(n),\ut\right)$) consisting of weight vectors such that
\begin{align*}
   v_1&=Q_{k,\ut},&w_1&=\wt{Q}_{k,\ut},&(v_i,w_j)&=0,\,i\neq j, &L(\ba)\times L(\ltrans{\ba}^{-1})\bigg(\sum_{j=1}^d v_j\otimes w_j\bigg)=\sum_{j=1}^d v_j\otimes w_j,\,\,\ba\in\GL(\bC).
\end{align*}
Then \eqref{eq:Iinv} implies that $\cI_{k,\ut}$ equals a multiple of $\sum_{j=1}^d v_j\times w_j$. The linear functional of evaluating at $x=\begin{pmatrix}\bid_n&0\end{pmatrix}$ annihilates all the weight vectors in $V^{\OO\text{-HW}}_{\fH,k,\ut}, V^{\OO\text{-LW}}_{\fH,k,\ut}$ except $v_1,w_1$. Its evaluation at $\cI_{k,\ut}$ shows that
\begin{equation}\label{eq:dI}
   \dim\left(\GL(n),\ut\right)^{-1}\cdot \text{the value of $\cI_{k,\ut}$ under the pairing \eqref{eq:pairgl}}=\left<\phi_{Q_{k,\ut}},\phi_{\wt{Q}_{k,\ut}}\right>.
\end{equation}
Since $\omega_{2k}\left(\begin{pmatrix}0&-\bid_n\\\bid_n&0\end{pmatrix},1\right)$ acts by $(-i)^{\sum_{j=1}^n t_j}$ on $\phi_Q$ for all $Q\in V^{\OO\text{-LW}}_{\fH,k,\ut}$ due to Proposition~\ref{prop:phequiv}, the left hand side of \eqref{eq:dI} equals the right hand side of the identity in the statement of the proposition and we finish the proof.
\end{proof}

\subsection{Computing the integral involving $\cI_{k,\ut}$}
The integral on the right hand side of the identity in Proposition~\ref{prop:WI} is particularly convenient to compute.
\begin{prop}\label{prop:IGamma}
$$
    \int_{M_{n,2k}(\bR)}\cI_{k,\ut}(x,x)\,e^{-2\pi\,\Tr x\ltrans{x}}\,dx=2^{-\sum_{j=1}^nt_j}\pi^{-\sum_{j=1}^n t_j+nk}\,\frac{\prod_{j=1}^n\Gamma(t_j-j-k+n+1)}{\prod_{j=1}^n\Gamma(n-j+1)}.
$$
\end{prop}
\begin{proof}
We use $z$ as the coordinate of $M_{n,k}(\bC)$ and identify $M_{n,2k}(\bR)$ with $M_{n,k}(\bC)$ via $z=\bz(x)=u+iv$ for $x=\begin{pmatrix}u&v\end{pmatrix}$. Then
\begin{equation}\label{eq:IQxz}
   \int_{M_{n,2k}(\bR)}\cI_{k,\ut}(x,x)\,e^{-2\pi\,\Tr x\ltrans{x}}\,dx=2^{-nk}\int_{M_{n,k}(\bC)}\fQ_{k,\ut}\Big((\ltrans{z}\ol{z})_{\text{upper left $n\times n$-block}}\Big)\,e^{-2\pi\,\Tr z\ltrans{\ol{z}}}\,|dzd\ol{z}|.
\end{equation}
Every $z\in M_{n,k}(\bC)$ can be uniquely written as 
\begin{align*}
   z=&\,\varsigma\begin{pmatrix}r_1&w_{12}&\cdots&w_{1n}&w_{1,n+1}&\cdots&w_{1k}\\0&r_2&\cdots&w_{2n}&w_{2,n+1}&\cdots&w_{2k}\\\vdots&\vdots&\ddots&\vdots&\vdots&\ddots&\vdots\\0&0&\cdots&r_n&w_{n,n+1}&\cdots&w_{nk}\end{pmatrix},&\begin{array}{l}\varsigma\in \U(n,\bR),\,r_1,\dots,r_n\in\bR_{>0},\\w_{ij}\in\bC,\,1\leq i\leq n, i<j\leq k\end{array}.
\end{align*}
Let $\omega_{\U}$ be the Haar measure on $\U(n,\bR)$ such that the total volume is $1$. Then there exists a constant $C_{n,k}\in\bR_{\geq 0}$ such that
\begin{align*}
   |dzd\ol{z}|&=C_{n,k}\,r^{2n-1}_1r^{2n-3}_2\dots r_n\,|dr_1\,dr_2\dots dr_n\,dw\,d\ol{w}\wedge\omega_{\U}|,&dw=\bigwedge_{\substack{1\leq i\leq 2\\i<j\leq k}}dw_{ij}.
\end{align*}
By evaluating the integral of the Gaussian function $e^{-\Tr z\ltrans{\ol{z}}}$, we get
$$
   C_{n,k}=2^n\,(2\pi)^{\frac{n(n+1)}{2}}\prod_{j=1}^n\Gamma(n-j+1)^{-1}.
$$
By definition
$$
   \fQ_{k,\ut}\Big((\ltrans{z}\ol{z})_{\text{upper left $n\times n$-block}}\Big)=r^{2t_1-2k}_1r^{2t_2-2k}_2\dots r^{2t_n-2k}_n.
$$
Hence,
\begin{align*}
   \eqref{eq:IQxz}=&\,2^{-nk}\,C_{n,k}\int_{\bR^n_{>0}}r^{2t_1-2k+2n-1}_1r^{2t_2-2k+2n-3}_2\dots r^{2t_n-2k+1}_n\,e^{-2\pi\sum_{j=1}^n r^2_j}\,dr_1\dots dr_n\\
   &\times\prod_{\substack{1\leq i\leq 2\\2\leq j\leq k,\,i<j}}\int_{\bC}e^{-2\pi\,|w_{ij}|^2}\,dw_{ij}\,d\ol{w}_{ij}\\
   =&\,2^{-nk}\,C_{n,k}\,\prod 2^{-n}(2\pi)^{-\sum_{j=1}^n t_j+nk-\frac{n(n+1)}{2}}\prod_{j=1}^n\Gamma(t_j-k+n+1-j)
\end{align*}
\end{proof}

\subsection{The archimedean zeta integral}
We sum up results in previous sections towards computing $Z_\infty\left(f_{k,\ut},v^\vee_{\ut},v_{\ut}\right)$. Combining Proposition~\ref{prop:ZW}, \eqref{eq:WFD} and Propositions~\ref{prop:WI}, \ref{prop:IGamma}, we get
\begin{equation*}
   Z_\infty\left(f_{k,\ut},v^\vee_{\ut},v_{\ut}\right)=\frac{i^{-\sum_{j=1}^nt_j+nk}2^{-2\sum_{j=1}^nt_j+nk}\pi^{-\sum_{j=1}^nt_j+nk}}{\dim\left(\GL(n),\ut\right)}\,\frac{\dim\,\lambda(\ut,k)}{d(\cD_{\ut},dg)}\frac{\prod_{j=1}^n\Gamma(t_j-j-k+n+1)}{\prod_{j=1}^n\Gamma(n-j+1)}.
\end{equation*}
The dimension of the $\OO(2k)$-representation $\lambda(\ut,k)$ equals the dimension of the irreducible algebraic $\SO(2k)$-representation of highest weight $(t_1-k,\dots,t_n-k,0,\dots,0)$. The Weyl dimension formula tells that
$$
   \dim\left(\SO(2k),(a_1,\dots a_k)\right)=\prod_{1\leq i< j\leq k}\frac{(a_i-a_j-i+j)(a_i+a_j+2k-i-j)}{(-i+j)(2k-i-j)},
$$
so			
\begin{align*}
   \dim\,\lambda(\ut,k)=&\frac{\prod\limits_{1\leq i< j\leq n}(t_i-t_j-i+j)(t_i+t_j-i-j)}{\prod\limits_{\substack{1\leq i<j\leq k\\i\leq n}}(-i+j)(2k-i-j)}\prod\limits_{1\leq j\leq n}(t_j-1)\frac{\Gamma(t_j-j+k-n)}{\Gamma(t_j-j-k+n+1)}\\
   =&\frac{\prod\limits_{1\leq i< j\leq n}(t_i-t_j-i+j)(t_i+t_j-i-j)}{2^{2nk-n(n+2)}\pi^{-n^2}\,\Gamma_{2n}(k)}\prod\limits_{1\leq j\leq n}(t_j-j)\frac{\Gamma(t_j-j+k-n)}{\Gamma(t_j-j-k+n+1)}.
\end{align*}

According to the formulas on the formal degree in \cite[Corollary of Lemma 23.1]{HC-HIII}(see also \cite[Lemma 2.3]{HII-FD}), we have
\begin{align*}
   d(\cD_{\ut},dg)=&\,\frac{2^{\frac{-n}{2}-\dim G(\bR)/K+\mr{rk}\, G(\bR)/K}(2\pi)^{\#\text{ positive roots of }G}}{\vol(T,dT)\,\vol(K,dK)^{-1}}\,\left|\prod_{\alpha\succ 0}\left<H_\alpha,(t_1-1,\dots,t_n-n)\right>\right|\\
   =&\,\frac{2^{-n^2}\pi^{-\frac{n^2+n}{2}}}{\prod_{j=1}^n\Gamma(j)}\prod_{1\leq i< j\leq n}(t_i-t_j-i+j)(t_i+t_j-i-j)\prod_{j=1}^n(t_j-j).
\end{align*}
Combining all the formulas proves our formula for the archimedean zeta integral.
\begin{thm}\label{thm:Main}
For integers $t_1\geq t_2\geq\dots\geq t_n\geq k\geq n+1$, let $f_{k,\ut}$ be the section inside the degenerate principal series $I(k)$ on $\Sp(4n,\bR)$ as defined in \eqref{eq:ourf}. Let $v_{\ut}$ be the highest weight vector in the lowest $K_G$-type of the holomorphic discrete series $\cD_{\ut}$, and $v^\vee_{\ut}\in\cD^\vee_{\ut}$ be its dual vector. Then we have the following formula for the doubling archimedean zeta integral associated with $f_{k,\ut}$ and the matrix coefficient $\left<g\cdot v^\vee_{\ut},v_{\ut}\right>$,
\begin{align*}
   Z_\infty\left(f_{k,\ut},v^\vee_{\ut},v_{\ut}\right)=\frac{i^{-\sum_{j=1}^n t_j+nk}2^{-2\sum_{j=1}^nt_j-nk+2n^2+2n}\pi^{-\sum_{j=1}^nt_j+nk+\frac{3n^2+n}{2}}}{\Gamma_{2n}(k)\,\dim\left(\GL(n),\ut\right)}\,\prod_{j=1}^n\Gamma(t_j-j+k-n).
\end{align*}
\end{thm}

\bibliographystyle{halpha}
\bibliography{BiStd}

\begin{thebibliography}{MVW87}

\bibitem[B\"85]{BoDouble}
Siegfried B\"{o}cherer.
\newblock \"{U}ber die {F}unktionalgleichung automorpher {$L$}-{F}unktionen zur
  {S}iegelschen {M}odulgruppe.
\newblock {\em J. Reine Angew. Math.}, 362:146--168, 1985.

\bibitem[B{\"o}c85]{BoDiff}
S~B{\"o}cherer.
\newblock {\"U}ber die {F}ourier-{J}acobi-{E}ntwicklung {S}iegelscher
  {E}isensteinreihen {II}.
\newblock {\em Mathematische Zeitschrift}, 189(1):81--110, 1985.

\bibitem[BS00]{BS}
S.~B\"{o}cherer and C.-G. Schmidt.
\newblock {$p$}-adic measures attached to {S}iegel modular forms.
\newblock {\em Ann. Inst. Fourier (Grenoble)}, 50(5):1375--1443, 2000.

\bibitem[Coa91]{Coates}
John Coates.
\newblock Motivic {$p$}-adic {$L$}-functions.
\newblock In {\em {$L$}-functions and arithmetic ({D}urham, 1989)}, volume 153
  of {\em London Math. Soc. Lecture Note Ser.}, pages 141--172. Cambridge Univ.
  Press, Cambridge, 1991.

\bibitem[CPR89]{CoPerrin}
John Coates and Bernadette Perrin-Riou.
\newblock On {$p$}-adic {$L$}-functions attached to motives over {${\bf Q}$}.
\newblock In {\em Algebraic number theory}, volume~17 of {\em Adv. Stud. Pure
  Math.}, pages 23--54. Academic Press, Boston, MA, 1989.

\bibitem[EHLS16]{EHLS}
Ellen Eischen, Michael Harris, Jianshu Li, and Christopher Skinner.
\newblock $p$-adic {$L$}-functions for unitary groups, part {II}: zeta-integral
  calculations, 2016, http://arxiv.org/abs/1602.01776.

\bibitem[EW16]{EisWan}
Ellen Eischen and Xin Wan.
\newblock $p$-adic {E}isenstein series and {$L$}-functions of certain cusp
  forms on definite unitary groups.
\newblock {\em J. Inst. Math. Jussieu}, 15(3):471--510, 2016.

\bibitem[Gar84]{GaPull}
Paul~B. Garrett.
\newblock Pullbacks of {E}isenstein series; applications.
\newblock In {\em Automorphic forms of several variables ({K}atata, 1983)},
  volume~46 of {\em Progr. Math.}, pages 114--137. Birkh\"{a}user Boston,
  Boston, MA, 1984.

\bibitem[Gar08]{GaAzint}
Paul Garrett.
\newblock Values of {A}rchimedean zeta integrals for unitary groups.
\newblock In {\em Eisenstein series and applications}, volume 258 of {\em
  Progr. Math.}, pages 125--148. Birkh\"{a}user Boston, Boston, MA, 2008.

\bibitem[Har81]{Ha81}
Michael Harris.
\newblock Special values of zeta functions attached to {S}iegel modular forms.
\newblock {\em Ann. Sci. \'{E}cole Norm. Sup. (4)}, 14(1):77--120, 1981.

\bibitem[Har86]{HaBun}
Michael Harris.
\newblock Arithmetic vector bundles and automorphic forms on {S}himura
  varieties. {II}.
\newblock {\em Compositio Math.}, 60(3):323--378, 1986.

\bibitem[Har08]{HaSim}
Michael Harris.
\newblock A simple proof of rationality of {S}iegel-{W}eil {E}isenstein series.
\newblock In {\em Eisenstein series and applications}, volume 258 of {\em
  Progr. Math.}, pages 149--185. Birkh\"{a}user Boston, Boston, MA, 2008.

\bibitem[HC76]{HC-HIII}
Harish-Chandra.
\newblock Harmonic analysis on real reductive groups. {III}. {T}he
  {M}aass-{S}elberg relations and the {P}lancherel formula.
\newblock {\em Ann. of Math. (2)}, 104(1):117--201, 1976.

\bibitem[HII08]{HII-FD}
Kaoru Hiraga, Atsushi Ichino, and Tamotsu Ikeda.
\newblock Formal degrees and adjoint {$\gamma$}-factors.
\newblock {\em J. Amer. Math. Soc.}, 21(1):283--304, 2008.

\bibitem[Ibu99]{Ib}
Tomoyoshi Ibukiyama.
\newblock On differential operators on automorphic forms and invariant
  pluri-harmonic polynomials.
\newblock {\em Comment. Math. Univ. St. Paul.}, 48(1):103--118, 1999.

\bibitem[JV79]{JV}
Hans~Plesner Jakobsen and Mich\`ele Vergne.
\newblock Restrictions and expansions of holomorphic representations.
\newblock {\em J. Funct. Anal.}, 34(1):29--53, 1979.

\bibitem[KR90]{KRPoles}
Stephen~S. Kudla and Stephen Rallis.
\newblock Poles of {E}isenstein series and {$L$}-functions.
\newblock In {\em Festschrift in honor of {I}. {I}. {P}iatetski-{S}hapiro on
  theoccasion of his sixtieth birthday, {P}art {II} ({R}amat{A}viv, 1989)},
  volume~3 of {\em Israel Math. Conf. Proc.}, pages 81--110. Weizmann,
  Jerusalem, 1990.

\bibitem[KV78]{KVWeil}
M.~Kashiwara and M.~Vergne.
\newblock On the {S}egal-{S}hale-{W}eil representations and harmonic
  polynomials.
\newblock {\em Invent. Math.}, 44(1):1--47, 1978.

\bibitem[Li90]{LiThetaCoh}
Jian-Shu Li.
\newblock Theta lifting for unitary representations with nonzero cohomology.
\newblock {\em Duke Math. J.}, 61(3):913--937, 1990.

\bibitem[Liu15]{LiuD-U21}
Dongwen Liu.
\newblock Archimedean zeta integrals on {$U(2,1)$}.
\newblock {\em J. Funct. Anal.}, 269(1):229--270, 2015.

\bibitem[Liu16]{SLF}
Zheng Liu.
\newblock $p$-adic ${L}$-functions for ordinary families of symplectic groups.
\newblock {\em Preprint}, 2016, http://www.math.mcgill.ca/zliu/SLF IV.pdf.
\newblock To appear in {\it J. Inst. Math. Jussieu}.

\bibitem[LL16]{LiuD-Un1}
Bingchen Lin and Dongwen Liu.
\newblock Archimedean zeta integrals on {$U(n,1)$}.
\newblock {\em J. Number Theory}, 169:62--78, 2016.

\bibitem[LR05]{LapidRallis}
Erez~M. Lapid and Stephen Rallis.
\newblock {O}n the local factors of representations of classical groups.
\newblock In {\em Automorphic representations, {$L$}-functions and
  applications: progress and prospects}, volume~11 of {\em Ohio State Univ.
  Math. Res. Inst. Publ.}, pages 309--359. de Gruyter, Berlin, 2005.

\bibitem[LR18]{LRtz}
Zheng Liu and Giovanni Rosso.
\newblock Non-cuspidal {H}ida theory for {S}iegel modular forms and trivial
  zeros of $p$-adic ${L}$-functions.
\newblock {\em Preprint}, 2018, http://arxiv.org/abs/1803.10273.

\bibitem[MVW87]{MVW}
C.~M\oe{}glin, M.-F. Vign\'{e}ras, and J.-L. Waldspurger.
\newblock {\em Correspondances de {H}owe sur un corps p-adique}, volume 1291 of
  {\em Lecture Notes in Mathematics}.
\newblock Springer-Verlag, Berlin, 1987.

\bibitem[PSR87]{PSR}
I.~Piatetski-Shapiro and Stephen Rallis.
\newblock ${L}$-functions for the classical groups.
\newblock volume 1254 of {\em Lecture Notes in Mathematics}, pages 1--52.
  Springer-Verlag, Berlin, 1987.

\bibitem[Shi84]{ShDiff}
Goro Shimura.
\newblock On differential operators attached to certain representations of
  classical groups.
\newblock {\em Invent. Math.}, 77(3):463--488, 1984.

\bibitem[Shi90]{ShLie}
Goro Shimura.
\newblock {I}nvariant differential operators on {H}ermitian symmetric spaces.
\newblock {\em Ann. of Math.}, 132(2):237--272, 1990.

\bibitem[Shi00]{Sh00}
Goro Shimura.
\newblock {\em Arithmeticity in the theory of automorphic forms}, volume~82 of
  {\em Mathematical Surveys and Monographs}.
\newblock American Mathematical Society, Providence, RI, 2000.

\bibitem[Yam14]{Yam}
Shunsuke Yamana.
\newblock {$L$}-functions and theta correspondence for classical groups.
\newblock {\em Invent. Math.}, 196(3):651--732, 2014.

\end{thebibliography}
\end{document}